\documentclass[12pt]{article}
\usepackage{mathrsfs}
\usepackage{amsfonts}
\usepackage{amssymb,amsmath,amssymb}\usepackage{amsthm}
\usepackage{array}
\usepackage[dvips]{graphics}
\usepackage{graphicx}
\usepackage{epstopdf}
\usepackage{ae}
\usepackage{bm}
\usepackage{graphicx}
\usepackage{indentfirst}
\usepackage{cite}
\usepackage{lineno}
\usepackage{enumitem}
\usepackage{xcolor}

\def\sign{\mbox{\rm Sgn}}
\def\Proj{\mbox{\rm Proj}}
\def\col{\mbox{\rm col}}
\def\row{\mbox{\rm row}}

\def\And{\mbox{\rm ~and~}}

\def\For{\mbox{\rm ~for~}}

\def\({\mbox{\rm (}}\def\){\mbox{\rm )}}

\def\i{\mbox{\rm (\hspace{0.2mm}i\hspace{0.2mm})}\,}
\def\ii{\mbox{\rm (\hspace{-0.1mm}i\hspace{-0.2mm}i\hspace{-0.1mm})}\,}
\def\iii{\mbox{\rm (\hspace{-0.5mm}i\hspace{-0.3mm}i\hspace{-0.3mm}i\hspace{-0.5mm})}\,}

\def\ccup{\textstyle\bigcup}

\def\Sqcup{\textstyle\bigsqcup\limits}

\def\T{{\scriptscriptstyle T}}
\def\x{{\scriptscriptstyle {\bm x}}}

\def\sst{\scriptscriptstyle}

\theoremstyle{plain}
\newtheorem{theorem}{Theorem}[section]
\newtheorem{lemma}[theorem]{Lemma}
\newtheorem{proposition}[theorem]{Proposition}
\newtheorem{example}[theorem]{Example}

\newtheorem{definition}[theorem]{Definition}

\newtheorem{remark}[theorem]{Remark}
\newtheorem{corollary}[theorem]{Corollary}

\begin{document}


\title{Bijections on $r$-Shi and $r$-Catalan Arrangements}
\author{Houshan Fu\thanks{Supported by Hunan Provincial Innovation Foundation for Postgraduate (CX2018B215)}\\
\small School of Mathematics\\
\small Hunan University\\
\small Changsha, Hunan, China\\
\small fuhoushan@hnu.edu.cn
\and
Suijie Wang\thanks{Supported by NSFC 11871204}\\
\small School of Mathematics\\
\small Hunan University \\
\small Changsha, Hunan, China\\
\small wangsuijie@hnu.edu.cn
\and
Weijin Zhu\\
\small School of Electronic Information and Electrical Engineering\\
\small Shanghai Jiao Tong University \\
\small Shanghai, China\\
\small weijinzhu@sjtu.edu.cn
}

\date{}
\maketitle
\begin{abstract}
Associated with the $r$-Shi arrangement and $r$-Catalan arrangement in $\Bbb{R}^n$, we introduce a cubic matrix for each region to establish two bijections in a uniform way. Firstly, the positions of minimal positive entries in column slices of the cubic matrix  will give a bijection from regions of the $r$-Shi arrangement to $O$-rooted labeled $r$-trees. Secondly, the numbers of positive entries in column slices of the cubic matrix will give  a bijection from regions of the $r$-Catalan arrangement to pairings of permutation and $r$-Dyck path.  Moreover, the numbers of positive entries in row slices of the cubic matrix will recover
the Pak-Stanley labeling, a celebrated bijection from regions of the $r$-Shi arrangement to $r$-parking functions.

\vspace{2ex}

\end{abstract}

\section{Concepts and Backgrounds}
This paper aims to establish two  bijections: from regions of the $r$-Shi arrangement to $O$-rooted labeled $r$-trees, and from regions of the $r$-Catalan arrangement to pairings of permutation and $r$-Dyck path. To this end, we introduce a cubic matrix for each region to read the combinatorial information from the region.

{\bf A hyperplane arrangement} $\mathcal{A}$ is a finite collection of hyperplanes in a vector space $V$, see \cite{Orlik-Terao,Stanley1}. When $V$ is a real space, the set $V\setminus \cup_{H\in \mathcal{A}}H$ consists of finitely many connected components, called {\bf regions} of $\mathcal{A}$. Denote by $\mathcal{R}(\mathcal{A})$ the set of regions of $\mathcal{A}$. For any positive integers $r$ and $n$, the {\bf $r$-Shi arrangement} $\mathcal{S}_n^r$ in $\mathbb{R}^n$ consists of the following hyperplanes
\[
\mathcal{S}_n^r:\; x_i-x_j=-r+1,-r+2,\ldots, 0,1,\ldots, r, \quad  1\le i<j\le n.
\]
The case of $r=1$ is the classical {\bf Shi arrangement} $\mathcal{S}_n$ introduced by Shi \cite{Shi} in 1986. Shi further obtained the number of regions of $\mathcal{S}_n^r$.
\begin{theorem}{\rm \cite{Shi1}}\label{Shi}
For any positive integers $r$ and $n$, the number of regions of $\mathcal{S}_n^r$ is
\[
|\mathcal{R}(\mathcal{S}_n^r)|=(rn+1)^{n-1}.
\]
\end{theorem}

Let $O=\{o_1,\ldots, o_r\}$ and $V=\{v_1,\ldots, v_n\}$ be two disjoint sets of labeled vertices. First introduced by Harary and Palmer \cite{Harary-Palmer} in 1968,  an {\bf $O$-rooted labeled $r$-tree} $T$ on $O\cup V$  is a graph having the property: there is a valid rearrangement $\nu=(v_{i_1},\ldots, v_{i_n})$ of vertices $v_1,\ldots, v_n$,  such that each $v_{i_j}$ with $j\in [n]$ is adjacent to exactly $r$ vertices in $\{o_1,\ldots, o_r, v_{i_1},\ldots, v_{i_{j-1}}\}$ and, moreover, these $r$ vertices are themselves mutually adjacent in $T$. Section 3 will be devoted to some characterizations of the $O$-rooted labeled $r$-trees. Denote by $\mathcal{T}^r_n$ the set of all $O$-rooted labeled $r$-trees. In the case of $r=1$, write $\mathcal{T}_n=\mathcal{T}_n^1$, whose members are called {\bf $O$-rooted labeled trees}. The size of $\mathcal{T}^r_n$  has been counted by Foata \cite{Foata}, Beineke and Pippert \cite{Beineke-Pippert,Beineke-Pippert1}, Gainer-Dewar and Gessel \cite{Andrew-Gessel} etc., which extends the Cayley formula $|\mathcal{T}_n|=(n+1)^{n-1}$ of \cite{Cayley}.

\begin{theorem}\label{rooted labeled r-tree}{\rm \cite{Foata,Beineke-Pippert,Beineke-Pippert1}}
For any positive integers $r$ and $n$, the cardinality of $\mathcal{T}^r_n$ is
\[
|\mathcal{T}^r_n|=(rn+1)^{n-1}.
\]
\end{theorem}
Closely related to $O$-rooted labeled $r$-trees and $r$-Shi arrangement, the {\bf $r$-parking function of length $n$} is a sequence $\alpha=(\alpha_1,\ldots,\alpha_n)\in \Bbb{Z}_{\ge 0}^n$ such that the monotonic rearrangement $a_1\le a_2\le\cdots\le a_n$ of the numbers $\alpha_1,\ldots,\alpha_n$ satisfies $a_i\le r(i-1)$. Denote by $\mathcal{P}_n^r$ the set of all $r$-parking functions of length $n$. In the case of $r=1$, write $\mathcal{P}_n=\mathcal{P}_n^1$ whose members are called {\bf parking functions of length $n$}.  Explored by Pitman and Stanley \cite{Stanley-Pitman}, Yan \cite{Yan1} etc.,  the cardinality of $\mathcal{P}_n^r$ is exactly the same as $\mathcal{R}(\mathcal{S}_n^r)$ and $\mathcal{T}^r_n$, namely,
\[
|\mathcal{P}_n^r|=(rn+1)^{n-1}.
\]

Naturally we may ask if there are some bijections among $\mathcal{R}(\mathcal{S}_n^r)$, $\mathcal{T}^r_n$, and $\mathcal{P}_n^r$. A celebrated bijection $\mathcal{R}(\mathcal{S}_n^r)\to \mathcal{P}_n^r$ (abbreviation for `from $\mathcal{R}(\mathcal{S}_n^r)$ to $\mathcal{T}_n^r$') is the Pak-Stanley labeling which was first suggested  by I. Pak in the case of $r=1$,  and extended to general $r$  by R. P. Stanley \cite{Stanley2,Stanley3}. Later, relevant to the Pak-Stanley labeling, many results on bijections $\mathcal{R}(\mathcal{S}_n^r)\to \mathcal{P}_n^r$ have been obtained, see \cite{Athanasiadis-Linusson,Beck,Duarte-Oliveira,Mazin,Postnikov-Stanley,
Stanley1,Stanley2,Stanley3} etc.. For the bijection $\mathcal{P}_n^r\to\mathcal{T}_n^r$, currently we just know that it can be established by a composition of three other bijections given in \cite{Pak-Postnikov} by I. Pak and A. Postnikov. In the case of $r=1$, bijections $\mathcal{P}_n\to \mathcal{T}_n$  have been well studied  since 1968,  see  \cite{Francon,Foata-Riordan,Kreweras,Knuth,Riordan,Schutzenberger} etc..   To the best of our knowledge, no explicit bijection $\mathcal{R}(\mathcal{S}_n^r)\to \mathcal{T}_n^r$  has been established, which is exactly our motivation of this paper. Our first main result is to establish a bijection $\mathcal{R}(\mathcal{S}_n^r)\to \mathcal{T}_n^r$, see Theorem \ref{Main-1}. To this end, we will introduce a cubic matrix for $r$-Shi arrangement, which will also let us define the Pak-Stanley labeling in an easy way, see Theorem \ref{Main-2}.

Surprisingly, the cubic matrix method can be applied to the {\bf $r$-Catalan arrangement} $\mathcal{C}_n^r$ in $\mathbb{R}^n$, the collection of hyperplanes
\[
\mathcal{C}_n^r:\quad x_i-x_j=0,\pm 1,\ldots,\pm r,\quad \mbox{for}\;1\le i<j\le n.
\]
When $r=1$, denote $\mathcal{C}_n=\mathcal{C}_n^1$, called the {\bf Catalan arrangement}. The number of regions of $\mathcal{C}_n^r$ was first obtained by Athanasiadis \cite{Athan2004} in 2004.
\begin{theorem}\label{Catalan}{\rm \cite{Athan2004}}
For any positive integers $r$ and $n$, the number of regions of $\mathcal{C}_n^r$ is
\[
|\mathcal{R}(\mathcal{C}_n^r)|=n!C(n,r)=\frac{n!}{rn+1}\binom{rn+n}{n}.
\]
\end{theorem}
In Theorem \ref{Catalan}, the number $C(n,r)=\frac{1}{rn+1}\binom{rn+n}{n}$ is called the {\bf Fuss-Catalan number} or {\bf Raney number}, which counts the number of the $r$-Dyck paths of length $n$. As written in \cite{Koshy-book}, the Fuss-Catalan number was first studied by Fuss \cite{Fuss1791} in 1791, forty-seven years before Catalan investigated the parenthesization problem, see \cite{Forres, Graham, He, Liu, Mlot, Mlot2013, Penson, Koshy-book} for more results on the Fuss-Catalan number. The paper \cite{Mansour2008} presented several combinatorial structures which are counted by Fuss-Catalan numbers. In the case of $r=1$, $C(n,r)=C_n=\frac{1}{n+1}\binom{2n}{n}$ is the {\bf Catalan number}, see \cite{Stanley4} for a complete investigation on the Catalan number. Dyck path has many generalizations that have been widely studied in the past, see \cite{Cameron-Mcleod, Duchon, Fukukawa, Imaoka, Labelle, Labelle-Yeh, Ma, Mansour, Rukavicka}. As a generalization of Dyck path, a {\bf $r$-Dyck path of length $n$} is a lattice path in the $x$-$y$ plane moving from $(0,0)$ to $(n,rn)$ with steps $(1,0)$ and $(0,1)$ and never going above the line $y=rx$.  Denote by $\mathcal{D}_n^r$ the collection of all $r$-Dyck paths of length $n$ and $\mathcal{D}_n=\mathcal{D}_n^1$ the set of all {\bf Dyck paths} of length $n$.  In 1989, Krattenthaler \cite{Kratten} obtained the number of $r$-Dyck paths of length $n$.
\begin{theorem}\label{Dyck}{\rm \cite{Kratten}}
For any positive integers $r$ and $n$, the cardinality of $\mathcal{D}_n^r$ is
\[
|\mathcal{D}_n^r|=C(n,r).
\]
\end{theorem}
As our second main result,  we will establish a bijection $\mathcal{R}(\mathcal{C}_n^r)\to \mathfrak{S}_n\times \mathcal{D}_n^r$ in Theorem \ref{Main-3} via the cubic matrix defined for the $r$-Catalan arrangement, which will extend the bijection defined in \cite[p. 69]{Stanley1}.

\section{Main Results}
Our first main result  is a bijection  $\mathcal{R}(\mathcal{S}_n^r)\to \mathcal{T}_n^r$, which will be stated in Section 2.1 and proved in Section 4.  The second main result is a bijection $\mathcal{R}(\mathcal{C}_n^r)\to \mathfrak{S}_n\times\mathcal{D}_n^r$ and will be  given in Section 2.2.
\subsection{Bijection  $\mathcal{R}(\mathcal{S}_n^r)\to \mathcal{T}_n^r$}
By introducing  a cubic matrix for $r$-Shi arrangement, in this section we establish a bijection $\mathcal{R}(\mathcal{S}_n^r)\to \mathcal{T}_n^r$ and present a straightforward way to view the Pak-Stanley labelling. Given a region $\Delta\in \mathcal{R}(\mathcal{S}_n^r)$ and a representative ${\bm x}=(x_1,x_2,\ldots,x_n)\in \Delta$, define the $\mathbf{cubic\; matrix}$ $C_{\bm x}= \big(c_{ijk}({\bm x})\big)\in \Bbb{R}^{n\times n \times r}$ to be
\begin{equation}\label{cubic-matrix}
c_{ijk}({\bm x}) =
\begin{cases}
x_i-x_j-k,& \mbox{if}\; i<j;  \\
0 ,         & \mbox{if}\; i=j;  \\
x_i-x_j-k+1 , & \mbox{if}\; i>j,
\end{cases}
\end{equation}
which is an $r$-tuple of square matrices as the index $k$ running from $1$ to $r$. For any $i,j\in [n]$, let
\[
\row_i(C_{\bm x})=\big(c_{ijk}(\bm x)\big)_{j\in [n],k\in [r]}\quad \And \quad \col_j(C_{\bm x})=\big(c_{ijk}(\bm x)\big)_{i\in [n],k\in [r]},
\]
called the {\bf $i$-th row slice} and {\bf $j$-th column slice} of $C_{\bm x}$ respectively. Note that each hyperplane $H\in \mathcal{S}_n^r$ is exactly defined by the equation $H: c_{ijk}({\bm x})=0$ for some $i,j$ and $k$, and all points of $\Delta$ lie in the same side of $H$ since $\Delta\cap H=\emptyset$. It follows that $c_{ijk}({\bm x})$ has the same sign for all ${\bm x}\in \Delta$, namely, $\sign\big(c_{ijk}({\bm x})\big)$ is independent of the choice of representatives ${\bm x}\in \Delta$ and can be denoted by
\begin{equation}\label{r-sign-func}
\sign_{ijk}(\Delta)=\sign\big(c_{ijk}(\bm x)\big).
\end{equation}
Then $\sign(\Delta)= \big(\sign_{ijk}(\Delta)\big)$ automatically defines a bijection
\begin{equation}\label{r-sign-bijection}
\sign: \mathcal{R}(\mathcal{S}_n^r)\to \{\sign(\Delta)\mid \Delta\in \mathcal{R}(\mathcal{S}_n^r)\}.
\end{equation}
The symbol ${\bm x}$ is understand as either a point of $\Bbb{R}^n$ or indeterminate depending on its meaning in the context.
\begin{definition}\label{definition}
Let $O=\{o_1,\ldots, o_r\}$ and $V=\{v_1,\ldots, v_n\}$ be two disjoint sets of labeled vertices. Given a region $\Delta\in \mathcal{R}(\mathcal{S}_n^r)$ and  ${\bm x}\in \Delta$, for any $j\in [n]$, let $f(v_j)=(f_1(v_j),\ldots,f_r(v_j))\in (O\cup V)^r$ be defined recursively as follows,
\begin{itemize}
\item [\i] if all entries of $\col_j(C_{\bm x})$ are nonpositive, let $p_j=0$ and
             \[f(v_j)=(o_1,o_2,\ldots,o_r),\]
\item [\ii] otherwise, $p_j\ne 0$ and $\col_j(C_{\bm x})$ has a unique minimal positive entry at $(p_j, q_j)$, let
\begin{equation*}\label{f-vector}
f(v_j)=\big(f_1(v_{p_j}),\ldots, f_{q_j-1}(v_{p_j}), f_{q_j+1}(v_{p_j}), \ldots, f_{r}(v_{p_j}), v_{p_j}\big),
\end{equation*}
\end{itemize}
and let the map $F: V\to {O\cup V\choose r}$ with
\[
F(v_j)=\{f_i(v_j)\mid i\in [r]\}.
\]
Define the graph $T_{\bm x}$ on the vertex set $O\cup V$ such that  the vertex $v_j$ and vertices in $F(v_j)$ form an $(r+1)$-clique for all $j\in [n]$.
\end{definition}

Below is the first main result of this paper, whose proof is highly nontrivial and will be given in Section 4.
\begin{theorem}\label{Main-1}
With the same notations as Definition {\rm \ref{definition}},  the following map is a bijection,
\begin{equation}\label{r-Psi}
\Psi_n^r :\mathcal{R}(\mathcal{S}_n^r)\rightarrow \mathcal{T}^{r}_{n}, \quad\quad \Psi_n^r(\Delta)=T_{\bm x}\quad {\rm ~for ~any~} {\bm x}\in \Delta.
\end{equation}
\end{theorem}
In the case of $r=1$, the statements of Definition \ref{definition} and Theorem \ref{Main-1} become quiet simple, see Corollary \ref{Main-1-1}.
\begin{corollary}\label{Main-1-1}
Let $O=\{o_1\}$ and $V=\{v_1,\ldots, v_n\}$ be two disjoint sets of labeled vertices. Given a region $\Delta$ of $\mathcal{S}_n$, for any ${\bm x}\in\Delta$,  define an $n\times n$ matrix $A_{\bm x}=\big(a_{ij}({\bm x})\big)$ with
\begin{equation*}
a_{ij}({\bm x}) =
\begin{cases}
x_i-x_j-1,& {\rm if}\; i<j;  \\
0 ,       & {\rm if}\; i=j;  \\
x_i-x_j , & {\rm if}\; i>j,
\end{cases}
\end{equation*}
and a graph $T_{\bm x}$ on  $O\cup V$ such that for each $j\in [n]$, $v_j$ is adjacent to $v_{p_j}$, where $p_j$ is defined as follows,
\begin{itemize}
\item[$\i$] if column $j$ of $A_{\bm x}$ has no positive entry, assume $p_j=0$ and $v_0=o_1$;
\item[$\ii$] otherwise, column $j$ of $A_{\bm x}$ has a unique minimal positive entry at row $p_j$.
\end{itemize}
Then $T_{\bm x}$ is an $O$-rooted labeled tree and independent of the choice of representatives ${\bm x}\in \Delta$. Moreover, the map $\Psi_n: \mathcal{R}(\mathcal{S}_n)\to \mathcal{T}_n$ with $\Psi_n(\Delta)=T_{\bm x}$ is a bijection.
\end{corollary}

In 1998, a celebrated bijection $\mathcal{R}(\mathcal{S}_n^r)\to \mathcal{P}_n^r$ was obtained by Stanley \cite{Stanley3} and called {\bf the Pak-Stanley labeling}, which is defined recursively as follows. Start with the base region $\Delta_0\in \mathcal{R}(\mathcal{S}_n^r)$ with
\[
\Delta_0: x_1>x_2>\cdots>x_n>x_1-1,
\]
whose labeling is assumed to be $\lambda(\Delta_0)=(0,\ldots,0)\in \Bbb{Z}_{\ge 0}^n$. Suppose $\Delta\in \mathcal{R}(\mathcal{S}_n^r)$ has been labeled by $\lambda(\Delta)\in \Bbb{Z}_{\ge 0}^n$, and an unlabeled region $\Delta'\in \mathcal{R}(\mathcal{S}_n^r)$ is separated from $\Delta$ by a unique hyperplane $H:c_{ijk}({\bm x})=0$. Then define the region $\Delta'$ to be labeled by $\lambda(\Delta')=\lambda(\Delta)+e_i$. Using the cubic matrix $C_{\bm x}$, Theorem 2.1 of \cite{Stanley3} can be restated as follows.
\begin{theorem}{\rm \cite{Stanley3}}\label{Main-2}
Given a region $\Delta$ of $\mathcal{S}_n^r$ and ${\bm x}\in \Delta$, for any $i\in [n]$, let
\begin{equation*}\label{lambda}
\lambda_i(\Delta)={\rm ~the~number~of~positive~signs~of~} \sign\big(\row_i(C_{\bm x})\big).
\end{equation*}
The following map is a bijection
\begin{equation*}
\lambda: \mathcal{R}(\mathcal{S}_n^r)\to \mathcal{P}_n^r,\quad \Delta\mapsto \lambda(\Delta)=\big(\lambda_1(\Delta),\ldots,\lambda_n(\Delta)\big).
\end{equation*}
\end{theorem}
\begin{proof}
Note that the base region is
\[
\Delta_0=\{{\bm y}\in \Bbb{R}^n\mid c_{ijk}({\bm y})<0,i,j\in [n],k\in [r]\}.
\]
If the region $\Delta$ is separated from $\Delta_0$ by the hyperplane $H:c_{ijk}({\bm y})=0$, then ${\bm x}\in \Delta$ implies $c_{ijk}({\bm x})>0$. From the definition of the Pak-Stanley labeling, it is easily seen that $\lambda_i(\Delta)$ is the number of the hyperplanes $H:c_{ijk}({\bm y})=0$ separating $\Delta$ from $\Delta_0$. Namely, $\lambda_i(\Delta)$ is the number of positive entries in the $i$-th row slice of $C_{\bm x}$.
\end{proof}

\begin{remark}
Theoretically, the compositions of our bijection $(\Psi_n^r)^{-1}:\mathcal{T}_n^r\to \mathcal{R}(\mathcal{S}_n^r)$ in Theorem {\rm \ref{Main-1}} and the Pak-Stanley labeling  $\lambda: \mathcal{R}(\mathcal{S}_n^r)\to \mathcal{P}_n^r$ in Theorem {\rm \ref{Main-2}}  will produce a bijection $\mathcal{T}_n^r\to \mathcal{P}_n^r$, while it seems to be highly complicated and difficult to be stated explicitly.
\end{remark}

\subsection{Bijection $\mathcal{R}(\mathcal{C}_n^r)\to \mathfrak{S}_n\times\mathcal{D}_n^r$}
In this section, we will establish a bijection $\mathcal{R}(\mathcal{C}_n^r)\to \mathfrak{S}_n\times\mathcal{D}_n^r$. Similar as \cite[p. 68]{Stanley1}, the permutation group $\mathfrak{S}_n$ acts on $\Bbb{R}^n$ by permuting coordinates, i.e., if $\pi\in \mathfrak{S}_n$, for ${\bm x}=(x_1,\ldots,x_n)\in \Bbb{R}^n$ we have
\[
\pi({\bm x})=(x_{\pi(1)},\ldots, x_{\pi(n)}).
\]
Given a region $\Delta\in \mathcal{R}(\mathcal{C}_n^r)$ and ${\bm x}\in \Delta$,  there is a unique permutation $\pi_{\sst  \Delta}\in \mathfrak{S}_n$, independent of the choice of ${\bm x}\in \Delta$, such that
\begin{equation*}\label{permutation}
x_{\pi_{\sst  \Delta}(1)}>\cdots>x_{\pi_{\sst  \Delta}(n)}.
\end{equation*}
Note that $\mathcal{R}(\mathcal{C}_n^r)$ is $\mathfrak{S}_n$-invariant, i.e., for any $\pi\in \mathfrak{S}_n$ and $\Delta\in \mathcal{R}(\mathcal{C}_n^r)$, we have
\[
\pi(\Delta)=\{\pi({\bm x})\mid {\bm x}\in \Delta\}\in \mathcal{R}(\mathcal{C}_n^r).
\]
For $\pi\in \mathfrak{S}_n$, denote by
\[
\mathcal{R}_\pi(\mathcal{C}_n^r)=\big\{\Delta\in \mathcal{R}(\mathcal{C}_n^r)\mid \pi_{\sst  \Delta}=\pi\big\}.
\]
In particular, let
\[
\mathcal{R}_{\bm 1}(\mathcal{C}_n^r)=\big\{\Delta\in \mathcal{R}(\mathcal{C}_n^r)\mid \pi_{\sst  \Delta}={\bm 1} {\rm~is~the~identity~permutation}\big\}.
\]
It is clear that $\pi$ is a bijection from $\mathcal{R}_{\pi}(\mathcal{C}_n^r)$ to $\mathcal{R}_{\bm 1}(\mathcal{C}_n^r)$ and
\[
\mathcal{R}(\mathcal{C}_n^r)=\Sqcup_{\pi\in \mathfrak{S}_n}\mathcal{R}_\pi(\mathcal{C}_n^r).
\]
To obtain the bijection $\mathcal{R}(\mathcal{C}_n^r)\to \mathfrak{S}_n\times\mathcal{D}_n^r$, it is enough to establish a bijection $\mathcal{R}_{\bm 1}(\mathcal{C}_n^r)\to \mathcal{D}_n^r$. Given a region $\Delta\in \mathcal{R}_{\bm 1}(\mathcal{C}_n^r)$ and a representative ${\bm x}=(x_1,x_2,\ldots,x_n)\in \Delta$, define the $\mathbf{cubic\; matrix}$ $D_{\bm x}= \big(d_{ijk}({\bm x})\big)\in \Bbb{R}^{n\times n \times r}$ to be
\begin{equation*}\label{catalan-matrix}
d_{ijk}({\bm x}) =
\begin{cases}
x_i-x_j-k,& \mbox{if}\; i\ne j;  \\
0 ,         & \mbox{if}\; i=j;
\end{cases}
\end{equation*}
Similar as before, each hyperplane $H\in \mathcal{C}_n^r$ is exactly defined by the equation $H: d_{ijk}({\bm x})=0$ for some $i,j\in [n]$ with $i\ne j$ and $k\in [r]$. So we still have that $\sign\big(d_{ijk}({\bm x})\big)$ is independent of the choice of representatives ${\bm x}\in \Delta$.

For any $r$-Dyck path $P\in \mathcal{D}_n^r$, if the vertical line $x=i-\frac{1}{2}$ intersects $P$ at the $y$-coordinate $h_i(P)$,  the sequence  ${\bm h}(P)=(h_1(P),\ldots, h_n(P))\in \Bbb{Z}^n$ is nondecreasing and satisfies $0\le h_i(P)\le r(i-1)$,  called the {\bf height sequence} of $P$. Conversely, it is clear that any nondecreasing sequence ${\bm h}=(h_1,\ldots, h_n)$ with $0\le h_i\le r(i-1)$ uniquely determines a $r$-Dyck path $P$ of length $n$ such that ${\bm h}(P)={\bm h}$. Indeed, the height sequence of a $r$-Dyck path is also a $r$-parking function. Now we are ready to give the bijection $\mathcal{R}_{\bm 1}(\mathcal{C}_n^r)\to \mathcal{D}_n^r$. Given any region $\Delta\in \mathcal{R}_{\bm 1}(\mathcal{C}_n^r)$ and ${\bm x}\in \Delta$, let
${\bm h}(\Delta)=\big(h_1(\Delta),\ldots,h_n(\Delta)\big)$  be a sequence defined by
\begin{equation}\label{h_j}
h_j(\Delta)={\rm ~the~number~of~positive~signs~of~} \sign\big(\col_j(D_{\bm x})\big), \quad j\in [n].
\end{equation}
As we shall see in Theorem \ref{Main-3}, the sequence ${\bm h}(\Delta)$ is exactly the height sequence of a $r$-Dyck path of length $n$, say $P_{\sst \Delta}$, which defines the bijection
\begin{equation}\label{R_1}
\mathcal{R}_{\bm 1}(\mathcal{C}_n^r)\to \mathcal{D}_n^r,\quad \Delta\mapsto P_{\sst \Delta}.
\end{equation}
Now for any region $\Delta\in \mathcal{R}(\mathcal{C}_n^r)$, we have $\Delta\in \mathcal{R}_{\pi_{\Delta}}(\mathcal{C}_n^r)$ and $\Delta'=\pi_{\sst \Delta}(\Delta)\in \mathcal{R}_{\bm 1}(\mathcal{C}_n^r)$.  By abuse of notations, denote by $P_{\sst \Delta}$ the corresponding $r$-Dyck path $P_{\sst \Delta'}$ obtained from the above bijection \eqref{R_1}, namely $P_{\sst \Delta}=P_{\pi_\Delta(\Delta)}$ for any $\Delta\in \mathcal{R}(\mathcal{C}_n^r)$. Below is our second main result.
\begin{theorem}\label{Main-3}
For any positive integers $r$ and $n$, the following map is a bijection,
\begin{equation*}\label{Phi}
\Phi_n^r: \mathcal{R}(\mathcal{C}_n^r)\rightarrow \mathfrak{S}_n\times \mathcal{D}_n^r, \quad \Phi_n^r(\Delta)=(\pi_{\sst \Delta}, P_{\sst\Delta}).
\end{equation*}
\end{theorem}
\begin{proof}
Notice from Theorem \ref{Catalan} and \ref{Dyck} that the both $\mathcal{R}(\mathcal{C}_n^r)$ and $\mathfrak{S}_n\times \mathcal{D}_n^r$ have the same cardinality $n!C(n,r)$.  By the above arguments, it is enough to show that the map defined in \eqref{R_1} is injective. For any $\Delta\in\mathcal{R}_{\bm 1}(\mathcal{C}_n^r)$ and
$\bm x=(x_1,\ldots,x_n)\in\Delta$, we have $x_1>x_2>\cdots>x_n$. It is easily seen from the definition of the cubic matrix $D_{\bm x}= \big(d_{ijk}({\bm x})\big)\in \Bbb{R}^{n\times n \times r}$ that
\begin{itemize}
  \item[(a)] $d_{ijk}(\bm x)>0$ implies $i<j$;
  \item[(b)] if $i<j<j'$, then $d_{ijk}(\bm x)>0$ implies $d_{ij'k}(\bm x)>0$ since $d_{ij'k}(\bm x)>d_{ijk}(\bm x)$.
\end{itemize}
Note from the definition of \eqref{h_j} that for $j\in[n]$,
\begin{equation}\label{h}
h_j({\Delta})=\#\big\{(i,k)\in[n]\times[r]\mid d_{ijk}(\bm x)>0\big\}.
\end{equation}
The properties (a) and (b) imply $h_j(\Delta)\le r(j-1)$ for any $j\in [n]$ and $h_1(\Delta)\le h_2(\Delta)\le\cdots\le h_n(\Delta)$ respectively. So ${\bm h}(\Delta)=\big(h_1(\Delta),\ldots,h_n(\Delta)\big)$  is a height sequence of a $r$-Dyck path of length $n$, i.e., the map given in \eqref{R_1} is well-defined. Next we prove the injectivity of the map in \eqref{R_1} by contradiction. Suppose $\Delta$ and $\Omega$ are two distinct regions in  $\mathcal{R}_{\bm 1}(\mathcal{C}_n^r)$ with ${\bm h}(\Delta)={\bm h}(\Omega)$ and let ${\bm x\in \Delta}$ and ${\bm y}\in \Omega$.  From $\Delta\ne \Omega$, we have a minimal index $j\in[n]$ such that the hyperplane $H: d_{ijk}({\bm z})=0$ separates $\Delta$ from $\Omega$. Assume
\[
d_{ijk}({\bm x})=x_i-x_j-k>0 \quad \And\quad  d_{ijk}({\bm y})=y_i-y_j-k<0.
\]
Since $h_j(\Delta)=h_j(\Omega)$, from \eqref{h} there must exist a pairing $(i',k')\ne (i,k)$ such that
\[
d_{i'jk'}({\bm x})=x_{i'}-x_j-k'<0 \quad\And\quad d_{i'jk'}({\bm y})=y_{i'}-y_j-k'>0.
\]
By property (a), we have $i,i'<j$. If $i=i'$, we have $k'>x_i-x_j>k$ since $d_{ijk}({\bm x})>0$ and $ d_{i'jk'}({\bm x})<0$  and $k>y_i-y_j>k'$ since $d_{ijk}({\bm y})<0$ and $ d_{i'jk'}({\bm y})>0$, which is a contradiction. If $i<i'$, we have $k>y_i-y_j>y_{i'}-y_j>k'$ since $d_{ijk}({\bm y})<0$ and $ d_{i'jk'}({\bm y})>0$. Consider the hyperplane $H: d_{ii'(k-k')}({\bm z})=0$. We have
\begin{eqnarray*}
d_{ii'(k-k')}({\bm x})&=&d_{ijk}({\bm x})-d_{i'jk'}({\bm x})>0,\\
d_{ii'(k-k')}({\bm y})&=&d_{ijk}({\bm y})-d_{i'jk'}({\bm y})<0,
\end{eqnarray*}
which means that the hyperplane $H: d_{ii'(k-k')}({\bm z})=0$ separates $\Delta$ from $\Omega$, a contradiction to the minimality of the index $j$. By similar arguments as the case $i<i'$, we can obtain a contradiction for the case $i>i'$. So we can conclude that the map in \eqref{R_1} is injective, which completes the proof.
\end{proof}
It is easily seen that Theorem \ref{Main-3} not only extends the bijection $ \mathcal{R}(\mathcal{C}_n)\to \mathfrak{S}_n\times\mathcal{D}_n$ defined in \cite[page 69]{Stanley1}, but make it more straightforward with the help of the cubic matrix. Below is an example to illustrate the construction of the Dyck path from a region in the case of $r=1$.
\begin{example}\label{Dyck-path}
Let $\Delta\in \mathcal{R}(\mathcal{C}_6)$ be the region
\begin{eqnarray*}
\Delta=\left\{{\bm x}=(x_1,\ldots,x_6)\in \Bbb{R}^6\left|
\begin{array}{l}
x_4>x_3>x_6>x_1>x_2>x_5,\\
x_3-x_1>1,\;x_1-x_2>1,\\
x_4-x_6<1,\;x_6-x_1<1,\;x_2-x_5<1.
\end{array}
\right.\right\}.
\end{eqnarray*}
It is obvious that $\pi_{\sst \Delta}=436125\in \mathfrak{S}_6$ and  for $1\le i<j\le 6$,
\begin{eqnarray*}
\Delta'=\pi_{\sst \Delta}(\Delta)=\left\{{\bm x}=(x_1,\ldots,x_6)\in \Bbb{R}^6\left|
\begin{array}{l}
x_1>x_2>x_3>x_4>x_5>x_6,\\
x_2-x_4>1,\;x_4-x_5>1,\\
x_1-x_3<1,\;x_3-x_4<1,\;x_5-x_6<1.
\end{array}
\right.\right\}.
\end{eqnarray*}
It follows that for any ${\bm x}\in \Delta'$,
\begin{eqnarray*}
\sign(c_{ij1}({\bm x}))=\left\{\begin{array}{ll}
-,&{\rm if}\; (i,j)\in {\big\{}(1,2), (1,3),(2,3),(3,4),(5,6){\big\}};\vspace{.2cm}\\
+,& {\rm if}\; (i,j)\in {\big\{}(1,4), (2,4), (1,5),(2,5),(3,5),(4,5),(1,6),(2,6),(3,6),(4,6){\big \}}.
\end{array}
\right.
\end{eqnarray*}
So we have ${\bm h}(\Delta')=(0,0,0,2,4,4)$, which is the height sequence of the Dyck path $P_{\sst\Delta}=P_{\sst\Delta'}$. Namely, $\Phi_n^1(\Delta)=(\pi_\Delta, P_{\sst\Delta})$ with $\pi_\Delta=436125$ and $P_{\sst\Delta}=$ the red path of {\rm Figure\,-\ref{figure2}}.
\begin{figure}[ht]
\centering
\includegraphics[width=5cm,height=5cm]{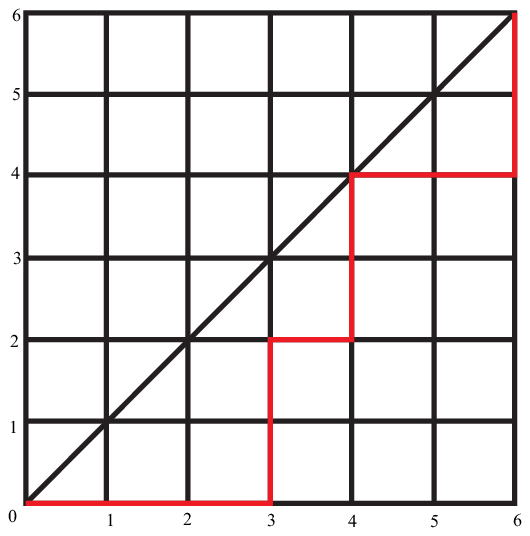}
\hspace{2cm}\caption{The Dyck path $D_{\sst \Delta}$}
\label{figure2}
\end{figure}
\end{example}

\section{$O$-Rooted Labeled $r$-Trees}
Preparing for Theorem \ref{Main-1}, we give some characterizations on $O$-rooted labeled $r$-trees in this section. For the structural integrity, below we restate the definition of $O$-rooted labeled $r$-trees following from Foata \cite{Foata} in 1971.
\begin{definition}\label{r-tree}{\rm \cite{Foata}}
Let $O=\{o_1,\ldots, o_r\}$ and $V=\{v_1,\ldots, v_n\}$ be two disjoint sets of labeled vertices.  An {\bf $O$-rooted labeled $r$-tree} $T$ on $O\cup V$  is a graph having the property: there is a valid rearrangement $\nu=(v_{i_1},\ldots, v_{i_n})$ of vertices $v_1,\ldots, v_n$,  such that each $v_{i_j}$ with $j\in [n]$ is adjacent to exactly $r$ vertices in $\{o_1,\ldots, o_r, v_{i_1},\ldots, v_{i_{j-1}}\}$ and, moreover, these $r$ vertices are themselves mutually adjacent in $T$. Let
\[
F_{\T}^\nu(v_{i_j})=\{v\mid v {\rm ~is~adjacent ~to~} v_{i_j} {\rm ~in~} T\} \cap \{o_1,\ldots, o_r, v_{i_1},\ldots, v_{i_{j-1}}\},
\]
whose members are called {\bf fathers of $v_{i_j}$ in $T$ under $\nu$}.
\end{definition}
\begin{remark}\label{remark}
{\rm Note from the above definition that the father set of $v_{i_1}$ in $T$ under $\nu$ is the root set $O$, so vertices of $O$ are mutually adjacent in $T$. In the case of $r=1$, for any ordinary tree $T$ on the labeled vertices $O\cup V$,  suppose that $\nu$ is a rearrangement having the property: $v_i$ is ahead of $v_j$ in $\nu$ whenever $d_\T(v_i,o_1)<d_\T(v_j,o_1)$ as distances of two vertices in $T$. Obviously such $\nu$ always exists and is valid for defining $T$ as an $O$-rooted labeled tree.}
\end{remark}
To make the above definition more clear,  Propositions \ref{proposition-1}-\ref{proposition-3} are characterizations on valid rearrangements and father sets, which might have been obtained by others in the literature but not noticed by us yet. Indeed, the $r$-trees have been characterized exactly to be the maximal graphs with a given treewidth in \cite{NOP2008}, and  the chordal graphs all of whose maximal cliques are the same size $r+1$ and all of whose minimal clique separators are also all the same size $r$ in \cite{Patil1986}.

\begin{proposition}\label{proposition-1}
Let $O=\{o_1,\ldots, o_r\}$ and $V=\{v_1,\ldots, v_n\}$ be two disjoint sets of labeled vertices, and $T$ an $O$-rooted labeled $r$-tree on $O\cup V$. For each $i\in [n]$, the father set $F_\T^\nu(v_i)$ is independent of the choice of valid rearrangements $\nu$ for $T$, and denoted by $F_\T(v_i)$.
\end{proposition}
\begin{proof}
Without loss of generality, we may assume that $\epsilon=(v_1,\ldots, v_n)$ is a valid rearrangement for $T$. Given a new valid rearrangement $\nu$ of vertices $v_1,\ldots, v_n$ for $T$, suppose $s$ is the minimal number such that $F_{\T}^{\epsilon}(v_s)\ne F_{\T}^{\nu}(v_s)$. If $o_j\in F_{\T}^{\epsilon}(v_s)\setminus F_{\T}^{\nu}(v_s)$ for some $j\in [r]$, then $o_j\notin F_{\T}^{\nu}(v_s)$ implies that $v_s$ is not adjacent to $o_j$ in $T$, a contradiction to $o_j\in F_{\T}^{\epsilon}(v_s)$. If $v_t\in F_{\T}^{\epsilon}(v_s)\setminus F_{\T}^{\nu}(v_s)$, then we have $t<s$ since $v_t\in F_{\T}^{\epsilon}(v_s)$ and $v_s\notin F_{\T}^{\epsilon}(v_t)=F_{\T}^{\nu}(v_t)$ by the minimality of $s$. Note the fact that $v_s$ is adjacent to $v_t$ in $T$, a contradiction to $v_s\notin F_{\T}^{\nu}(v_t)$ and $v_t\notin F_{\T}^{\nu}(v_s)$. Hence, $F_{\T}^{\epsilon}(v_i)= F_{\T}^{\nu}(v_i)$ for all $i\in [n]$.
\end{proof}

\begin{proposition}\label{proposition-2}
Let $O=\{o_1,\ldots, o_r\}$ and $V=\{v_1,\ldots, v_n\}$ be two disjoint sets of labeled vertices, and $T$ an $O$-rooted labeled $r$-tree on $O\cup V$. A rearrangement $\nu$ of vertices $v_1,\ldots, v_n$ is valid for $T$ if and only if  $v_s$ is ahead of $v_t$ in $\nu$ whenever $v_s\in F_\T(v_t)$.
\end{proposition}
\begin{proof}
The sufficiency is obvious from the definition of $O$-rooted labeled $r$-tree. To prove the necessity, we may assume that $\epsilon=(v_1,\ldots, v_n)$ is a valid rearrangement for $T$. Proposition \ref{proposition-1} implies $F_{\T}^{\epsilon}(v_i)=F_{\T}(v_i)$. Now suppose that $\nu$ is a rearrangement such that $v_s$ is ahead of $v_t$ in $\nu$ whenever
$v_s\in F_{\T}^{\epsilon}(v_t)$. Let $G_{\T}^{\nu}(v_i)$ consist of those vertices $o_1,\ldots o_r$ who are adjacent to  $v_i$ in $T$, and vertices $v_1,\ldots v_n$ who are adjacent to $v_i$ in $T$ and ahead of $v_i$ in $\nu$. Immediately, we have $F_{\T}^{\epsilon}(v_i)\subseteq G_{\T}^{\nu}(v_i)$ for all $i\in [n]$ and $G_{\T}^{\nu}(v_i)\cap O=F_{\T}^{\epsilon}(v_i)\cap O$. To obtain the necessity, i.e., $\nu$ is valid for $T$, it is enough to show $F_{\T}^{\epsilon}(v_i)=G_{\T}^{\nu}(v_i)$. Suppose $G_{\T}^{\nu}(v_t)\ne F_{\T}^{\epsilon}(v_t)$ and $v_s\in G_{\T}^{\nu}(v_t)\setminus F_{\T}^{\epsilon}(v_t)$ for some $s,t\in [n]$. By the definition of $G_{\T}^{\nu}(v_t)$, $v_s\in G_{\T}^{\nu}(v_t)$ implies that $v_s$ is ahead of $v_t$ in $\nu$ and adjacent to $v_t$ in $T$. From the assumption of $\nu$, we have $v_t\notin F_{\T}^{\epsilon}(v_s)$. Note that $v_s$ and $v_t$ are adjacent, a contradiction to $v_t\notin F_{\T}^{\epsilon}(v_s)$ and $v_s\notin F_{\T}^{\epsilon}(v_t)$.
\end{proof}

\begin{proposition}\label{proposition-3}
Let $O=\{o_1,\ldots,o_r\}$ and $V=\{v_1,\ldots,v_n\}$ be two disjoint sets of labeled vertices, and $F: V\to {O\cup V\choose r}$. There is an $O$-rooted labeled $r$-tree $T$ on $O\cup V$ with $F(v_i)=F_\T(v_i)$ for all $i\in [n]$  if and only if $F$ satisfies the following properties:
\begin{itemize}
\item[{\rm (a)}] if $v_{i_1}\in F(v_{i_2}), \ldots, v_{i_{j-1}}\in F(v_{i_j})$ for some $i_1,\ldots, i_j\in [n]$, then $v_{i_j}\notin F(v_{i_1})$;
\item[{\rm (b)}] if $F(v_j)\ne O$,  then there is a vertex $v_{i}\in  F(v_j)$ such that $|F(v_j)\cap F(v_i)|=r-1$.
\end{itemize}
Moreover, both the $r$-tree $T$ and the vertex $v_i$ in {\rm (b)} are unique.
\end{proposition}
\begin{proof}Let's prove the second part first. If $T$ is an $O$-rooted labeled $r$-tree on $O\cup V$ with $F(v_i)=F_\T(v_i)$, then  all vertices of $F(v_i)\cup \{v_i\}$ are mutually adjacent, which  exactly form all edges of $T$. So $T$ is uniquely determined by $F(v_i)=F_\T(v_i)$.  To prove the uniqueness of $v_i$ in (b), note that (a) implies $v_j\notin F(v_j)$ for all $j\in [n]$. Suppose there is another vertex $v_{i'}\in F(v_j)$ with $i'\ne i$ such that $|F(v_j)\cap F(v_{i'})|=r-1$. Then we have $v_i\in F(v_{i'})$ and $v_{i'}\in F(v_i)$, a contradiction to (a).

To prove the sufficiency of the first part, we may assume that $\epsilon=(v_1,\ldots, v_n)$ is a valid rearrangement for $T$.   From Proposition \ref{proposition-2},  if $v_j\in F(v_i)$, $v_j$ is ahead of $v_i$ in $\epsilon$, i.e., $j<i$. So if $v_{i_1}\in F(v_{i_2}), \ldots, v_{i_{j-1}}\in F(v_{i_j})$ for some $i_1,\ldots, i_j\in [n]$, then $i_1<i_j$ which implies $v_{i_j}\notin F(v_{i_1})$ and (a) holds. To prove (b), let $i$ be the largest number with $v_i\in F(v_j)$, $i<j$ obviously.   Suppose $v_{i'}\in F(v_j)\setminus F(v_i)$, then $i'<i<j$ and $v_i\notin F(v_{i'})$.  Note $v_i,v_{i'}\in F(v_j)$ and all vertices of  $F(v_j)$ are mutually adjacent, which is a contradiction to $v_{i'}\notin F(v_i)$ and $v_{i}\notin F(v_{i'})$. Thus we have $F(v_j)\setminus F(v_i)=\{v_i\}$, i.e., $|F(v_j)\cap F(v_i)|=r-1$. Moreover, for any  $v_{i'}\in F(v_j)$ with $i'<i$, we have at least $v_{i'}, v_i\notin F(v_{i'})$, i.e.,  $|F(v_j)\cap F(v_{i'})|\le r-2$.

To prove the necessity of the first part, from the assumption $v_{i_j}\notin F(v_{i_1})$ whenever $v_{i_1}\in F(v_{i_2}), \ldots, v_{i_{j-1}}\in F(v_{i_j})$ for some $i_1,\ldots, i_j\in [n]$, we have $v_{i_t}\notin F(v_{i_s})$ for $1\le s< t\le j$, which implies $v_{i_s}\ne v_{i_t}$, i.e., $i_s\ne i_t$ for $1\le s< t\le j$ since $v_{i_s}\in F(v_{i_{s+1}})$ and $v_{i_t}\notin F(v_{i_{s+1}})$, and $j\le n$ consequently. Suppose $v_{i_1}\in F(v_{i_2}), \ldots, v_{i_{k-1}}\in F(v_{i_k})$ for some $i_1,\ldots, i_k\in [n]$, where $k$ is maximal possible. The maximality of $k$ implies $v_{i_k}\notin F(v_{i})$ for all $i\in [n]$. Next we use induction on the size of $V$. When $|V|=1$, note that from (a), $v_1\in F(v_{1})$ produces $v_1\notin F(v_1)$, which forces $v_1\notin F(v_{1})$, i.e., $F(v_{1})=O$ and the result follows clearly. Let $V'=V\setminus \{v_{i_k}\}$ and $F': V'\to  {O\cup V'\choose r}$ with $F'(v_i)=F(v_i)$ for all $v_i\in V'$. It is clear that $(a)$ and $(b)$ holds for $F'$.  From the induction hypothesis, there is an $O$-rooted labeled $r$-tree $T'$ on $O\cup V'$ such that $F'(v_i)=F_{\T'}(v_i)$ with $v_i\in V'$. Given a valid rearrangement $\nu'$ for $T'$, let $T$ be a graph on the labeled vertex set $O\cup V$ obtained from $T'$ by adding the $r$ edges between $v_{i_k}$ and each vertex of $F(v_{i_k})$, and let $\nu=(\nu', v_{i_k})$. The case of $F(v_{i_k})=O$ is clear. Otherwise, from (b) we have $|F(v_{i_k})\cap F(v_i)|=r-1$ for some $v_i\in F(v_{i_k})$. Note all vertices of $F(v_i)\cup\{v_i\}$ are mutually adjacent in $T'$, which implies that all vertices of $F(v_{i_k})$ are also mutually adjacent in $T'$. Consequently, the graph $T$ is an $O$-rooted labeled $r$-tree and $\nu$ is a valid rearrangement for $T$.
\end{proof}

\section{Proof of Theorem \ref{Main-1}}
Roughly speaking, in Definition \ref{definition} the graph $T_{\bm x}$ is obtained by the following process,
\begin{equation*}\label{procedure}
{\bm x}\in \Delta \;\longrightarrow \;p_j, q_j({\rm if~} p_j\ne 0)\;\longrightarrow \;f(v_j)\;\longrightarrow F(v_j)\;\longrightarrow\; T_{\bm x}.
\end{equation*}
which requires that $p_j$, $q_j$(if $p_j\ne 0$), $f(v_j)$, $F(v_j)$, and $T_{\bm x}$ are well-defined for all $j\in [n]$, see Proposition \ref{r-well-defined}.

\begin{proposition}\label{r-well-defined}
With the same notations as Definition {\rm \ref{definition}}, the graph $T_{\bm x}$ is independent of the chioce of ${\bm x}\in \Delta$. Moreover,  $T_{\bm x}$ is an $O$-rooted labeled $r$-tree with $F_{\T_{\bm x}}(v_j)=F(v_j)$ for all $j\in [n]$, namely, the map $\Psi_n^r$ in \eqref{r-Psi} is well defined.
\end{proposition}
\begin{proof}
Firstly, we will show that $p_j$ and $q_j$(if $p_j\ne 0$) is independent of the choice of ${\bm x}\in \Delta$. Let $\Delta\in \mathcal{R}(\mathcal{S}_n^r)$ and $j\in [n]$. For the case $\i$ of Definition $\ref{definition}$, by (\ref{r-sign-func}) we have $\sign_{ijk}(\Delta)\ne +$ for all $i\in[n]$ and $k\in [r]$, which implies $p_j=0$ for all ${\bm x}\in \Delta$. For the case $\ii$ of Definition $\ref{definition}$, given any $(i,k)\ne (i',k')$ in $[n]\times [r]$ with $i>i'$, by routine calculations on entries of $C_{{\bm x},j}$ we have
\begin{equation*}\label{r-entry-comparison}
c_{ijk}({\bm x})-c_{i'jk'}({\bm x})
=\left\{
\begin{array}{ll}
c_{ii'(k-k')}({\bm x}), & \hbox{if $k>k'$, $i>j>i'$;} \\
-c_{i'i(k'-k+1)}({\bm x}), & \hbox{if $k\le k'$, $i>j>i'$;}\\
c_{ii'(k-k'+1)}({\bm x}), & \hbox{if $k\ge k'$, $i>i'>j$ or $j>i>i'$;}\\
-c_{i'i(k'-k)}({\bm x}), & \hbox{if $k< k'$, $i>i'>j$ or $j>i>i'$.}
\end{array}
\right.
\end{equation*}
For all ${\bm x}\in \Delta$, we have $\sign\big(c_{ijk}({\bm x})-c_{i'jk'}({\bm x})\big)=\sign\big(c_{ii's}({\bm x})\big)$ or $-\sign\big(c_{i'is}({\bm x})\big)$ for some $s\in [r]$, which means that both $\sign\big(c_{ijk}({\bm x})\big)$ and  $\sign\big(c_{ijk}({\bm x})-c_{i'jk'}({\bm x})\big)\ne 0$ are independent of the choice of ${\bm x}\in \Delta$. It implies that  for all ${\bm x}\in \Delta$, the $j$-th column slice $\col_j(C_{\bm x})$ has a unique minimal positive entry at the same position $(p_j,q_j)$, i.e., $(p_j,q_j)$ is independent of the choice of ${\bm x}\in \Delta$.

Secondly, we will prove that $f(v_j)$ and $F(v_j)$ is well defined.  Let
\begin{equation}\label{pi}
\pi:[n]\to \{0,1,\ldots,n\}\quad {\rm with}\quad \pi(j)=p_j,
\end{equation}
$\pi^{l+1}(j)=\pi(\pi^{l}(j))=p_{\pi^l(j)}$ for $l\ge 0$, and $\pi^0(j)=j$. Note that if $p_j\ne 0$, we have $c_{p_jjq_j}({\bm x})>0$ and
\[
c_{p_jjq_j}({\bm x})
=\left\{
\begin{array}{ll}
x_{p_j}-x_j-q_j, & \hbox{if $p_j<j$;} \\
x_{p_j}-x_j-q_j+1, & \hbox{if $p_j>j$,}
\end{array}
\right.
\]
which implies $x_{p_j}>x_j$ and $p_j\ne j$. Namely, we have $x_{\pi(j)}>x_j$ if $\pi(j)\ne 0$. There exists some $m\in[n]$ such that $\pi(j),\ldots, \pi^{m-1}(j)\ne 0$ and $\pi^m(j)=0$, otherwise $\pi(j), \ldots \pi^n(j)\ne 0$ and $x_j<x_{\pi(j)}<\cdots<x_{\pi^n(j)}$ which is obviously impossible for ${\bm x}=(x_1,\ldots,x_n)$. Moreover $x_j<x_{\pi(j)}<\cdots<x_{\pi^{m-1}(j)}$ implies that $j, \pi(j),\ldots, \pi^{m-1}(j)$ are mutually distinct. It follows from \i and \ii of Definition \ref{definition} that
\[
f(v_{p_{\pi^{m-2}(j)}})=f(v_{\pi^{m-1}(j)})=(o_1,\ldots,o_r),
\]
and for all $l=m-2,\ldots, 1, 0$, we have
\begin{equation}\label{f-recursion}
f(v_{\pi^{l}(j)})=\big(f_1(v_{\pi^{l+1}(j)}),\ldots, f_{q_{\pi^{l}(j)}-1}(v_{\pi^{l+1}(j)}), f_{q_{\pi^{l}(j)}+1}(v_{\pi^{l+1}(j)}), \ldots, f_{r}(v_{\pi^{l+1}(j)}), v_{\pi^{l+1}(j)}\big).
\end{equation}
Proceeding  $l$ from $m-2$ to $0$ step by step recursively,  finally we can obtain $f(v_j)$, which is well-defined consequently.  Moreover, note from the definition that for each $l=m-2,\ldots, 1, 0$,
\[
F(v_{\pi^{l}(j)}) =\big(F(v_{\pi^{l+1}(j)})\setminus f_{q_{\pi^{l}(j)}}(v_{\pi^{l+1}(j)})\big)\,\ccup\, \{v_{\pi^{l+1}(j)}\},
\]
i.e., $F(v_{\pi^{l}(j)})$ is obtained from $F(v_{\pi^{l+1}(j)})$ by removing the vertex $f_{q_{\pi^{l}(j)}}(v_{\pi^{l+1}(j)})$ and adding the vertex $v_{\pi^{l+1}(j)}$. So each $F(v_{\pi^{l}(j)})$ consists of $r$ members of vertices $o_1,\ldots, o_r, v_{\pi^{m-1}(j)},\ldots,v_{\pi^{l+1}(j)}$, which is of size $r$ since $j, \pi(j),\ldots, \pi^{m-1}(j)$ are mutually distinct and nonzero. In particular,  $|F(v_j)|=|F(v_{\pi^{0}(j)})|=r$ and $F(v_j)$ is well defined.

Finally, it remains to show that there exists uniquely an $O$-rooted labeled $r$-tree $T_{\bm x}$ on $O\cup V$ with $F_{\T_\x}(v_j)=F(v_j)$ for all $j\in [n]$. Recall the arguments of the above proof that if $\pi(j)\ne 0$, $F(v_j)$ consists of $r$ members of vertices $o_1,\ldots, o_r, v_{\pi^{m-1}(j)},\ldots,v_{\pi(j)}$, where $m$ is the smallest integer with $\pi^m(j)=0$ and $m\ge 2$. If $v_i\in F(v_j)$, then $\pi(j)=p_j\ne 0$ and
\[
v_i\in \{o_1,\ldots, o_r, v_{\pi^{m-1}(j)},\ldots,v_{\pi(j)}\},
\]
which follows $v_i=v_{\pi^l(j)}$, i.e., $i=\pi^l(j)$ for some positive integer $l\in [m-1]$. Now suppose $v_{i_1}\in F(v_{i_2}), \ldots, v_{i_{j-1}}\in F(v_{i_j})$ for some $i_1,\ldots, i_j\in [n]$. There exist some positive integers $l_1,\ldots, l_{j-1}$ such that
\[
i_1=\pi^{l_1}(i_2), \dots, i_{j-1}=\pi^{l_{j-1}}(i_j).
\]
We have $i_1=\pi^{l}(i_j)$ with $l=l_1+\cdots+l_{j-1}$, i.e., $v_{i_1}=v_{\pi^{l}(i_j)}$ which implies $v_{i_j}\notin F(v_{i_1})$. So the map $F$ satisfies the property (a) of Proposition $\ref{proposition-3}$. The property (b) is obvious since  $v_{p_j}\in F(v_j)$ and $|F(v_j)\cap F(v_{p_j})|=r-1$. The proof completes by Proposition $\ref{proposition-3}$.
\end{proof}
\begin{remark}{\rm  (1) From the above proof,  notations of Definition \ref{definition} can be written more precisely  as
\begin{equation}\label{notations}
p_j=p_j(\Delta),\quad q_j=q_j(\Delta), \quad f=f_\Delta,\quad F=F_\Delta, \And T_{\bm x}=T_\Delta,
\end{equation}
since they are all independent of the choice of ${\bm x}\in \Delta$.  (2) We also have the following observation
\begin{equation}\label{claim}
V=\{v_1,\ldots, v_n\}\nsubseteq \ccup_{j=1}^n F_\Delta(v_j),
\end{equation}
otherwise, we have $v_{i_1}\in F_\Delta(v_{i_2}),v_{i_2}\in F_\Delta(v_{i_3}),\ldots $ for an infinite sequence $i_1, i_2,\ldots$, which is a contradiction since $i_1=\pi^{l_1}(i_2), i_2=\pi^{l_2}(i_3),\ldots$ for some positive integers $l_1,l_2,\ldots$. }
\end{remark}
It is easily seen from \eqref{cubic-matrix} that for any $i,j,k\in [n]$ and $s,t\in [r]$ with $i>j>k$ and $s+t\le r$, we have the following facts on linear relations among the entries of the cubic matrix $C_{\bm x}$,
\begin{eqnarray*}
{\rm (F1)}& c_{ijs}({\bm x})+c_{jkt}({\bm x})=c_{ik(s+t-1)}({\bm x});\quad{\rm (F2)}& c_{iks}({\bm x})+c_{kjt}({\bm x})=c_{ij(s+t)}({\bm x});\\
{\rm (F3)}& c_{kis}({\bm x})+c_{ijt}({\bm x})=c_{kj(s+t-1)}({\bm x});\quad{\rm (F4)}& c_{kjs}({\bm x})+c_{jit}({\bm x})=c_{ki(s+t)}({\bm x});\\
{\rm (F5)}&c_{jks}({\bm x})+c_{kit}({\bm x})=c_{ji(s+t-1)}({\bm x});\quad{\rm (F6)}&c_{jis}({\bm x})+c_{ikt}({\bm x})=c_{jk(s+t)}({\bm x}).
\end{eqnarray*}

\begin{lemma}\label{r-sign-lemma}
If $p_j\ne 0$ and $q_j$ are defined as $\ii$ of Definition {\rm \ref{definition}}, then entries of $C_{\bm x}$ have the following sign relations,
\begin{equation*}
\sign\big({c_{ijk}}(\bm x)\big)
=\left\{
\begin{array}{ll}
\sign\big(c_{ip_j(k-q_j+1)}(\bm x)\big), & \hbox{if $q_j\le k$ and $(i,j,p_j)$ is even;} \\
-\sign\big(c_{p_ji(q_j-k)}(\bm x)\big), & \hbox{if $q_j> k$ and $(i,j,p_j)$ is even;}\\
-\sign\big(c_{p_ji(q_j-k+1)}(\bm x)\big), & \hbox{if $q_j\ge k$ and $(i,j,p_j)$ is odd;}\\
\sign\big(c_{ip_j(k-q_j)}(\bm x)\big), & \hbox{if $q_j<k$ and $(i,j,p_j)$ is odd,}
\end{array}
\right.
\end{equation*}
where $(i,j,p_j)$ is even if $i<j<p_j$, or $p_j<i<j$, or $j<p_j<i$, and odd otherwise.
\end{lemma}

\begin{proof}We prove the result in the case of $q_j\le k$ and $i<j<p_j$ whose arguments can be applied to other cases analogously.
When $i<j<p_j$, by the fact (F3) we have
\[
c_{p_jjq_j}(\bm x)=c_{ijk}(\bm x)-c_{ip_j(k-q_j+1)}(\bm x),
\]
which from the assumption is the unique minimal positive entry in the $j$-th column slice of $C_{\bm x}$. If $c_{ijk}(\bm x)$ is positive, by the unique minimality of $c_{p_jjq_j}(\bm x)$ we have $c_{ijk}(\bm x)>c_{p_jjq_j}(\bm x)$, which implies $c_{ip_j(k-q_j+1)}(\bm x)>0$. If $c_{ijk}(\bm x)$ is negative, by the positivity of $c_{p_jjq_j}(\bm x)$ we have $c_{ip_j(k-q_j+1)}(\bm x)<0$. Namely,  $\sign(c_{ijk}(\bm x))=\sign(c_{ip_j(k-q_j+1)}(\bm x))$.
\end{proof}

Let $\Proj_j:\Bbb{R}^n\to \Bbb{R}^{n-1}$ be the projection defined  by
\[
\Proj_j(x_1,\ldots,x_n)= (x_1,\ldots,x_{j-1},x_{j+1},\ldots,x_n).
\] It is clear that  $\Proj_j(\Delta)\in \mathcal{R}(\mathcal{S}_{n-1}^r)$ for any $\Delta\in \mathcal{R}(\mathcal{S}_n^r)$. Below is a key lemma to prove the injectivity of the map $\Psi_n^r$ of (\ref{r-Psi}).
\begin{lemma}\label{r-uniqueness}
Given $j'\in [n]$, for any $\Delta'\in \mathcal{R}(\mathcal{S}^r_{n-1})$, $i'\in\{0,1,\ldots,j'-1,j'+1,\ldots,n\}$, and $k'\in [r]$ (if $i'\ne 0$), there is at most one region $\Delta\in\mathcal{R}(\mathcal{S}^r_{n})$ such that $\Proj_{j'}(\Delta)=\Delta'$, $p_{j'}(\Delta)=i'$, $q_{j'}(\Delta)=k'$ (if $i'\ne 0$), and $j'\ne p_j(\Delta)$ for all $j\in [n]$, see {\rm (\ref{notations})} for  notations $p_j(\Delta)$ and $q_j(\Delta)$.
\end{lemma}
\begin{proof}We only consider the case of $j'=n$. For general $j'$, the arguments are analogous but more tedious.
Suppose two regions $\Delta, \Omega\in \mathcal{R}(\mathcal{S}_n^r)$ satisfying that $\Proj_n(\Delta)=\Proj_n(\Omega)=\Delta'$ and
\[
p_n(\Delta)=p_n(\Omega)=i',\quad q_n(\Delta)=q_n(\Omega)=k' ~({\rm if~} i'\ne 0),\; \And n\ne p_j(\Delta),~ p_j(\Omega)\For j\in [n].
\]
We will show $\sign(\Delta)=\sign(\Omega)$, which implies $\Delta=\Omega$ since the map $\sign: \mathcal{R}(\mathcal{S}_n^r)\to \{\sign(\Delta)\mid \Delta\in \mathcal{R}(\mathcal{S}_n^r)\}$ is a bijection by \eqref{r-sign-bijection}. Given ${\bm z}\in \Delta'$, let ${\bm x}\in \Delta$ and ${\bm y}\in \Omega$ such that $\Proj_n({\bm x})=\Proj_n({\bm y})={\bm z}$. By \eqref{r-sign-func} we have
$\sign_{ijk}(\Delta)=\sign(c_{ijk}({\bm x}))$ and $ \sign_{ijk}(\Omega)=\sign(c_{ijk}({\bm y}))$.
It is enough to show that  for all $i,j\in [n]$ and $k\in [r]$,
\begin{equation}\label{x=y}
\sign(c_{ijk}({\bm x}))=\sign(c_{ijk}({\bm y})).
\end{equation}
We first claim that for all $j\in [n]$,
\[
p_j(\Delta)=p_j(\Omega)=p_j\quad \And \quad q_j(\Delta)=q_j(\Omega)=q_j ~(\rm {if}~p_j\ne 0).
\]
Indeed, if $j=j'=n$, it is obvious from the assumptions. If $j\ne j'=n$, note  $n\ne p_j(\Delta)$ ($p_j(\Omega)$ resp.) for all $j\in [n]$. By the definitions of $p_j(\Delta)$ and $q_j(\Delta)$ ($p_j(\Omega)$ and $q_j(\Omega)$ resp.),   the minimal positive entry of $\col_j(C_{\bm x})$ ($\col_j(C_{\bm y})$ resp.) never appears in the $n$-th row slice of $C_{\bm x}$ ($C_{\bm y}$ resp.). It follows that $p_j(\Delta)=p_j(\Delta')=p_j(\Omega)$ and $q_j(\Delta)=q_j(\Delta')=q_j(\Omega)$ for $j\ne j'$, so the claim holds.
Notice that (\ref{x=y}) holds if $p_j=0$ since all entries of the $j$-th column slice of $C_{\bm x}$ are nonpositive, more precisely,  $\sign(c_{jjk}({\bm x}))=0$ and $\sign(c_{ijk}({\bm x}))=-$ if $i\ne j$.  Now we assume $p_j\ne 0$ and consider the following cases to prove (\ref{x=y}).
\begin{itemize}
\item[\i] For $i,j\in [n-1]$, since $\Proj_n({\bm x})=\Proj_n({\bm y})={\bm z}$, we have $c_{ijk}({\bm x})=c_{ijk}({\bm y})=c_{ijk}({\bm z})$. Thus \eqref{x=y} holds in this case.
\item[\ii] For $j=n$ and $i\in [n]$, note $p_n(\Delta)=p_n(\Omega)=i'\ne n$ and $q_n(\Delta)=q_n(\Omega)=k'$ (if $i'\ne 0$). If $k'\ge k$ and $i<i'<n$, the 3rd identity of Lemma \ref{r-sign-lemma} implies $\sign(c_{ink}(\bm x))=-\sign(c_{i'i(k'-k+1)}(\bm x))$ and $\sign(c_{ink}(\bm y))=-\sign(c_{i'i(k'-k+1)}(\bm y))$. Note from the case of \i above that $\sign(c_{i'i(k'-k+1)}(\bm x))=\sign(c_{i'i(k'-k+1)}(\bm y))$. Thus \eqref{x=y}  holds in this case. Other cases can be obtained by similar arguments.
\item[\iii] For $i=n$ and $j\in [n]$,  we have the following four cases.\\
(C-1). $q_j\ge k$ and $p_j<j<i=n$. From the 3rd identity of Lemma \ref{r-sign-lemma}, we have
\[
\sign(c_{njk}(\bm x))=-\sign(c_{p_jn(q_j-k+1)}(\bm x)) \quad\And\quad \sign(c_{njk}(\bm y))=-\sign(c_{p_jn(q_j-k+1)}(\bm y)).
\]

(C-2).  $q_j> k$ and $j<p_j<i=n$. From the 2nd identity of Lemma \ref{r-sign-lemma} we have
\[
\sign(c_{njk}(\bm x))=-\sign(c_{p_jn(q_j-k)}(\bm x)) \quad\And\quad \sign(c_{njk}(\bm y))=-\sign(c_{p_jn(q_j-k)}(\bm y)).
\]

(C-3).  $q_j< k$ and $p_j<j<i=n$. From the 4th identity of Lemma \ref{r-sign-lemma}, we have \[
\sign(c_{njk}(\bm x))=\sign(c_{np_j(k-q_j)}(\bm x)) \quad\And\quad \sign(c_{njk}(\bm y))=\sign(c_{np_j(k-q_j)}(\bm y)).
\]

(C-4).  $q_j\le  k$ and $j<p_j<i=n$.  From the 1st identity of  Lemma \ref{r-sign-lemma}, we have
\[
\sign(c_{njk}(\bm x))=\sign(c_{np_j(k-q_j+1)}(\bm x))\quad\And\quad \sign(c_{njk}(\bm y))=\sign(c_{np_j(k-q_j+1)}(\bm y)).
\]
It is obvious from \ii above that  \eqref{x=y} holds in cases (C-1) and (C-2). Next we will show \eqref{x=y} holds in (C-3) and (C-4) simultaneously by induction on $k$. For $k=1$, note that \eqref{x=y} holds if $p_j<j$ by (C-1), and also holds  if $q_j> k=1$ and $j<p_j$ by (C-2). In particular, we have $\sign(c_{n(n-1)1}(\bm x))=\sign(c_{n(n-1)1}(\bm y))$ obviously. The remainder case is $q_j=1$ and $j<p_j$. From the 1st identity of Lemma \ref{r-sign-lemma}, we have $\sign(c_{nj1}(\bm x))=\sign(c_{np_j1}(\bm x))$. Now consider $j$ from $n-2$ to 1 step by step as follows. For $j=n-2$, we have $p_j=n-1$ since $j<p_j$, and
$\sign(c_{n(n-2)1}(\bm x))=\sign(c_{n(n-1)1}(\bm x))=\sign(c_{n(n-1)1}(\bm y))=\sign(c_{n(n-2)1}(\bm y))$. For $j=n-3$, we have $p_j=n-1$ or $n-2$ since $j<p_j$, and
$\sign(c_{nj1}(\bm x))=\sign(c_{np_j1}(\bm x))=\sign(c_{np_j1}(\bm y))=\sign(c_{nj1}(\bm y))$. Continuing above steps,  we finally obtain $\sign(c_{nj1}(\bm x))=\sign(c_{nj1}(\bm y))$ for any $j\in [n]$, which proves  (\ref{x=y}) for $k=1$.   Now suppose (\ref{x=y}) holds in (C-3) and (C-4)  for $1,\ldots, k-1$. For general $k$,  since $k-q_j<k$, it is clear from the induction hypothesis  that \eqref{x=y} holds in (C-3). In particular, by (C-1) and (C-3) we have  $\sign(c_{n(n-1)k}(\bm x))=\sign(c_{n(n-1)k}(\bm y))$ for all $k\in [r]$. For the case  (C-4),  if $q_j> 1$, then $k-q_j+1<k$ and (\ref{x=y}) holds in this case by the induction hypothesis. So we only need to consider the case of $q_j=1$ and $j<p_j$ for (C-4),  i.e.,
\[
\sign(c_{njk}(\bm x))=\sign(c_{np_jk}(\bm x))\quad\And\quad \sign(c_{njk}(\bm y))=\sign(c_{np_jk}(\bm y)).
\]
Similar as the base case $k=1$, we may  consider $j$ from $n-2$ to 1 step by step, which can prove \eqref{x=y}   in this case. E.g. if $j=n-2<p_j$,  then $p_j=n-1$ and  we have $\sign(c_{n(n-2)k}(\bm x))=\sign(c_{n(n-1)k}(\bm x))=\sign(c_{n(n-1)k}(\bm y))=\sign(c_{n(n-2)k}(\bm y))$.  The proof of (\ref{x=y}) in case \iii completes.
\end{itemize}
\end{proof}

\begin{proposition}\label{r-injection}
The map $\Psi_n^r$ in {\rm (\ref{r-Psi})} is injective.
\end{proposition}
\begin{proof}
We will use induction on the dimension $n\ge 2$. For the induction base $n=2$, it is easy to see that under the map $\Psi_2^r$, all $2r+1$ regions of $\mathcal{S}_2^r$ are 1-1 corresponding to $O$-rooted labeled $r$-trees in $\mathcal{T}_2^r$. Suppose the result holds for $n-1$, i.e., if $T_\Delta'=\Psi_{n-1}^r(\Delta')=\Psi_{n-1}^r(\Omega')=T_\Omega'$ for any two regions $\Delta', \Omega'\in \mathcal{R}(\mathcal{S}_{n-1}^r)$, then we have $\Delta'=\Omega'$. Now suppose $\Delta, \Omega\in \mathcal{R}(\mathcal{S}_{n}^r)$ with $T_\Delta=T_\Omega=T\in \mathcal{T}_n^r$. By Proposition \ref{proposition-1} and Definition \ref{definition}, we have $F_\Delta(v_j)=F_\Omega(v_j)=F_\T(v_j)$ for all $j\in [n]$. We claim that $p_j(\Delta)= p_j(\Omega)$ and $q_j(\Delta)= q_j(\Omega)$ (if $p_j(\Delta)= p_j(\Omega)\ne 0$) for all $j\in [n]$, whose proof will be given later. From (\ref{claim}), we can take a vertex $v_{j'}\notin F_\Delta(v_j)=F_\Omega(v_j)$ for all $j\in [n]$. It is clear that
\begin{itemize}
\item[(a)] $j'\ne p_j(\Delta)$, $j'\ne p_j(\Omega)$ {\rm for all} $j\in [n]$.
\item[(b)] $p_{j'}(\Delta)=p_{j'}(\Omega)=i'$, $q_{j'}(\Delta)=q_{j'}(\Omega)=k'$ (if $i'\ne 0$).
\end{itemize}
Let $T'\in \mathcal{T}_{n-1}^r$ be the $O$-rooted labeled $r$-tree obtained from $T$ by removing the vertex $v_{j'}$ and the edges between $v_{j'}$ and vertices of $F_\T(v_{j'})$. Given ${\bm x}\in \Delta$ and ${\bm y}\in\Omega$, let
\begin{eqnarray*}
\Proj_{j'}(\Delta)=\Delta',\quad\Proj_{j'}({\bm x})={\bm x'},\quad \Proj_{j'}(\Omega)=\Omega',\quad \Proj_{j'}({\bm y})={\bm y'}.
\end{eqnarray*}
It follows from $\Proj_{j'}({\bm x})={\bm x'}$ that the cubic matrix $C_{\bm x'}\in \Bbb{R}^{(n-1)\times(n-1)\times r}$ is obtained from $C_{\bm x}\in \Bbb{R}^{n\times n\times r}$  by removing the $j'$-th column and row slices from $C_{\bm x}$. The above property (a) implies that for each $j\in [n]\setminus\{j'\}$, the minimal positivity entry $c_{p_jjq_j}(\bm x)$ of $\col_j(C_{\bm x})$  never appears in the $j'$-th row slice of $C_{\bm x}$. So for $j\in [n]\setminus\{j'\}$, the minimal positivity entry of $j$-th column slice $\col_j(C_{\bm x})$ appears in the same position as $\col_j(C_{\bm x'})$, i.e., $p_j(\Delta)=p_j(\Delta')=p_j$ and $q_j(\Delta)=q_j(\Delta')=q_j$ (if $p_j\ne0$). By the definition of $O$-rooted labeled $r$-tree in Definition \ref{definition}, we have $T'=T_{\bm x'}=\Psi_{n-1}^r(\Delta')$, and $T'=T_{\bm y'}=\Psi_{n-1}^r(\Omega')$ similarly. From the induction hypothesis, we obtain $\Delta'=\Omega'$, i.e., $\Proj_{j'}(\Delta)=\Proj_{j'}(\Omega)=\Delta'$. Combining with the above properties (a) and (b), Lemma \ref{r-uniqueness} implies $\Delta=\Omega$.

To prove the claim under the assumption $F_\Delta(v_j)=F_\Omega(v_j)=F_\T(v_j)$ for all $j\in [n]$, note
\[
p_j(\Delta)=0\quad\Leftrightarrow\quad F_\Delta(v_j)=O=F_\Omega(v_j)\quad\Leftrightarrow\quad p_j(\Omega)=0.
\]
Now for $p_j(\Delta)\ne 0$, we have $p_j(\Omega)\ne 0$ and from the Definition \ref{definition},
\[
v_{p_j(\Delta)}\in F_\Delta(v_j)=F_\Omega(v_j)\subseteq F_\Omega(v_{p_j(\Omega)})\ccup v_{p_j(\Omega)}.
\]
Suppose $p_j(\Delta)\ne p_j(\Omega)$. Then we have $v_{p_j(\Delta)}\in F_\Omega(v_{p_j(\Omega)})=F_\T(v_{p_j(\Omega)})$, and symmetrically $v_{p_j(\Omega)}\in F_\Delta(v_{p_j(\Delta)})=F_\T(v_{p_j(\Delta)})$, which is impossible by Proposition \ref{proposition-3}. Thus $p_j(\Delta)= p_j(\Omega)$ for all $j\in [n]$. Define $\pi:[n]\to \{0,1,\ldots,n\}$ with $\pi(j)=p_j(\Delta)= p_j(\Omega)$ as (\ref{pi}). If $\pi_j\ne 0$, we have shown that for some $m\ge 2$, we have $\pi(j),\ldots, \pi^{m-1}(j)\ne 0$ and $\pi^m(j)=0$. Next we prove $q_j(\Delta)= q_j(\Omega)$. Let's start with the convenient notations $({\bm u};i,u)$ and $[{\bm u};i,u]$ for an $r$-tuple ${\bm u}=(u_1,\ldots, u_r)$, an element $u$, and $i\in [r]$, where
\[
({\bm u};i,u)=(u_1,\ldots,u_{i-1},u_{i+1},\ldots, u_r,u)\;\And\;
[{\bm u};i,u]=\{u_1,\ldots,u_{i-1},u_{i+1},\ldots, u_r,u\}.
\]
It is easy to see that if $u_1,\ldots, u_r$ and $u$ are mutually distinct, then the following result holds
\begin{equation}\label{u-i-u}
[{\bm u};i,u]=[{\bm u};j,u]\quad \Leftrightarrow\quad i=j\quad \Leftrightarrow\quad({\bm u};i,u)=({\bm u};j,u).
\end{equation}
It follows from \i and \ii of Definition \ref{definition} that
\[
f_\Delta(v_{\pi^{m-1}(j)})=O=f_\Omega(v_{\pi^{m-1}(j)}),
\]
For any $l=m-2,\ldots, 1, 0$, it is clear from (\ref{f-recursion}) that
\begin{eqnarray*}
F_\Delta(v_{\pi^{l}(j)})&=&\left[f_\Delta(v_{\pi^{l+1}(j)}); q_{\pi^{l}(j)}(\Delta),v_{\pi^{l+1}(j)}\right],\\ F_\Omega(v_{\pi^{l}(j)})&=&\left[f_\Omega(v_{\pi^{l+1}(j)});q_{\pi^{l}(j)}(\Omega),v_{\pi^{l+1}(j)}\right],
\end{eqnarray*}
and $F_\Delta(v_{\pi^{l}(j)})=F_\Omega(v_{\pi^{l}(j)})$. Applying \eqref{u-i-u} and running $l$ from $m-2$ to $0$, we finally obtain
\[
q_j(\Delta)= q_j(\Omega)\quad\And \quad f_\Delta(v_j)=f_{\Omega}(v_j).
\]
\end{proof}
{\bf Proof of Theorem \ref{Main-1}.} We have obtained that the map $\Psi_n^r:\mathcal{R}(\mathcal{S}_n^r)\rightarrow \mathcal{T}_{n}^r$ of (\ref{r-Psi}) is well defined by Proposition \ref{r-well-defined} and injective by Proposition \ref{r-injection}, which is enough to conclude that $\Psi_n^r$ is a bijection from the fact that both $\mathcal{R}(\mathcal{S}_n^r)$ and $\mathcal{T}_{n}^r$ have the same cardinality by Theorem \ref{Shi} and Theorem \ref{rooted labeled r-tree}.

\begin{corollary}\label{r-existence}
Given $j'\in [n]$, for any $\Delta'\in \mathcal{R}(\mathcal{S}_{n-1}^r)$, $i'\in\{0,1,\ldots,j'-1,j'+1,\ldots,n\}$ and $k'\in [r]$ (if $i'\ne 0$), there is a region $\Delta\in \mathcal{R}(\mathcal{S}_{n}^r)$ such that  $\Proj_{j'}(\Delta)=\Delta'$, $p_{j'}(\Delta)=i'$, $q_{j'}(\Delta)=k'$ (if $i'\ne 0$), and $j'\ne p_j(\Delta)$ for all $j\in [n]$.
\end{corollary}
Note that Corollary \ref{r-existence} can be easily obtained from the surjectivity of $\Psi_n^r$. Recall Lemma \ref{r-uniqueness} that the uniqueness of the region $\Delta\in \mathcal{R}(\mathcal{S}_{n}^r)$ (if exists) is crucial to guarantee the injectivity of $\Psi_n^r$ in Proposition \ref{r-injection}. As a parallel situation, the existence of such region  $\Delta$ in Corollary \ref{r-existence} will produce a proof on the surjectivity of $\Psi_n^r$. However, similar as the Pak-Stanley labeling at the very beginning appeared in \cite{Stanley2}, we currently have no direct proof on the surjectivity of $\Psi_n^r$ without using Theorem \ref{Shi} and Theorem \ref{rooted labeled r-tree}.  So it would be of great interest to find a direct proof on Corollary \ref{r-existence}.

Recall arguments of  Proposition  \ref{r-injection} and  Lemma  \ref{r-uniqueness}, which provide an algorithm to construct  the region $\Delta$ of $r$-Shi arrangement from an $O$-rooted labeled $r$-tree $T$, i.e.,  the inverse map
\[
(\Psi_n^r)^{-1}: \mathcal{T}_n^r\to \mathcal{R}(\mathcal{S}_{n}^r), \quad\quad (\Psi_n^r)^{-1} (T)=\Delta.
\]
As a brief look, below is a small example to illustrate the bijection $\Psi_n^r$ by constructing the $O$-rooted labeled tree from a given region, and its inverse $(\Psi_n^r)^{-1}$ by constructing the region from a given $O$-rooted labeled tree.

\begin{example}\label{examp1}
{\rm For $n=3$ and $r=1$,  Figure\,-\ref{figure1} describes the complete correspondence between $\mathcal{R}(\mathcal{S}_3)$ and $\mathcal{T}_3$ under the map $\Psi_3$.
\begin{figure}[ht]
\centering
\includegraphics[width=10cm,height=7cm]{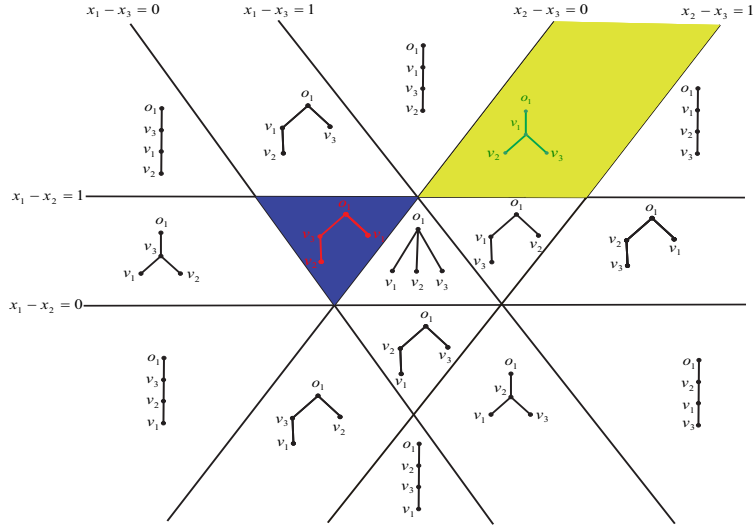}
\hspace{2cm}\caption{The bijection $\Psi_3:\mathcal{R}(\mathcal{S}_3)\to \mathcal{T}_3$}
\label{figure1}
\end{figure}
E.g., let $\Delta\in \mathcal{R}(\mathcal{S}_3)$ be the blue region in Figure\,-\ref{figure1} defined by
\[\Delta=\{0<x_1-x_2<1;\;0<x_1-x_3<1;\;x_2-x_3<0\},\]
and $\bm x=(0.2,-0.2,0)\in \Delta$. By Theorem \rm{\ref{Main-1}}, we have
\[
A_{\bm x}=\begin{pmatrix}

0    & -0.6 & -0.8\\
-0.4 & 0    & -1.2\\
-0.2 & 0.2  & 0
\end{pmatrix}
,
\]
and $p_1(\Delta)=0, p_2(\Delta)=3, p_3(\Delta)=0$. Namely, in the $O$-rooted labeled tree $T_{\bm x}=\Psi_3(\Delta)$, the fathers of $v_1,v_2,$ and $v_3$  are $o_1,v_3,$ and $o_1$ respectively, which exactly determines $T_{\bm x}$ to be the red tree in Figure \ref{figure1}. Conversely, let $T\in\mathcal{T}_3$ be the green tree in Figure\,-\ref{figure1} having $o_1,v_1,$ and $v_1$ as the fathers of $v_1,v_2,$ and $v_3$ respectively. If $\Omega=\Psi_3^{-1}(T)$, it follows from the definition of $T$ in Theorem \rm{\ref{Main-1}} that $p_1(\Omega)=0$, $p_2(\Omega)=1$, and $p_3(\Omega)=1$. Take the leaf $v_3$ of $T$. Let $T'$ be the tree obtained from $T$ by removing $v_3$ and the edge $v_3\sim v_1$, and $\Omega'=\Proj_3(\Omega)$. According to the proofs of Lemma \ref{r-uniqueness} and Proposition \ref{r-injection}, we have $\Psi_2^{-1}(T')=\Omega'$ and
\[
\sign(\Omega')=\begin{pmatrix} 0&+\\-&0\end{pmatrix}=\begin{pmatrix} \sign_{11}(\Omega)&\sign_{12}(\Omega)\\\sign_{21}(\Omega)&\sign_{22}(\Omega)\end{pmatrix}.
\]
Since $p_3(\Omega)=1$ and $p_1(\Omega)=0$, we have $\sign_{13}(\Omega)=+$ and $\sign_{31}(\Omega)=-$. By Lemma \ref{r-sign-lemma}, we have $\sign_{23}(\Omega)=\sign_{21}(\Omega)=-$ and $\sign_{32}(\Omega)=-\sign_{13}(\Omega)=-$. The sign matrix $\sign(\Omega)=\big(\sign_{ij}(\Omega)\big)_{3\times 3}$ exactly determines $\Omega$ to be the yellow region in Figure\,-\ref{figure1}.}
\end{example}


\end{document}